\documentclass[11pt]{amsart}

\usepackage{enumerate,url,amssymb,  mathrsfs}
\usepackage{graphicx}
\newtheorem{theorem}{Theorem}[section]

\newtheorem*{lemma*}{Lemma}
\newtheorem{proposition}[theorem]{Proposition}
\newtheorem{corollary}[theorem]{Corollary}
\newtheorem{conjecture}[theorem]{Conjecture}

\theoremstyle{definition}
\newtheorem{definition}[theorem]{Definition}
\newtheorem{example}[theorem]{Example}
\newtheorem{question}[theorem]{Question}

\theoremstyle{remark}
\newtheorem{remark}[theorem]{Remark}

\numberwithin{equation}{section}

\newcommand{\bydef}{\stackrel {\textnormal{def}}{=\!\!=} }

\def\XXint#1#2#3{{\setbox0=\hbox{$#1{#2#3}{\int}$}
\vcenter{\hbox{$#2#3$}}\kern-.5\wd0}}

\begin{document}
\baselineskip6mm
\vskip0.4cm
\title[A hunt for sharp $\,\mathscr L ^p\,$-estimates]{A hunt for sharp $\,\mathscr L ^p\,$-estimates and \\Rank-One Convex Variational Integrals}

\author{Kari\; Astala, \; Tadeusz\; Iwaniec,  \\ Istv\'an \;Prause, \; Eero \;Saksman}

\address{Department of Mathematics and Statistics,
University of Helsinki, Finland}
\email{kari.astala@helsinki.fi}
\thanks{Astala and Saksman were supported by the Academy of Finland and Center of Excellence in Analysis and Dynamics research. Iwaniec was supported by the NSF grant DMS-0800416 and Academy of
Finland grant 1128331. Prause  was supported by project 1266182 of the Academy of Finland}   

\address{Department of Mathematics, Syracuse University, Syracuse,
NY 13244, USA and Department of Mathematics and Statistics,
University of Helsinki, Finland}
\email{tiwaniec@syr.edu}

\address{Department of Mathematics and Statistics,
University of Helsinki, Finland}
\email{istvan.prause@helsinki.fi}

\address{Department of Mathematics and Statistics,
University of Helsinki, Finland}
\email{eero.saksman@helsinki.fi}


\maketitle
\begin{center}{\large{\textit{In celebration of Matti Vuorinen's 65-th birthday}}}\end{center}
\begin{abstract} Learning how to figure out sharp $\,\mathscr L^p\,$-estimates of nonlinear differential expressions, to prove and use them, is a fundamental part of the development of PDEs and Geometric Function Theory (GFT). Our survey presents, among what is known to date, some notable recent efforts and novelties made in this direction.  We focus attention here on the historic Morrey's Conjecture and Burkholder's martingale inequalities for stochastic integrals. Some of these topics have already been discussed by the present authors \cite{AIPS} and by Rodrigo Ba\~{n}uelos  \cite{Ba1}. Nevertheless, there is always something new to add.\\

\end{abstract}

\section{Introduction}
The $\,\mathscr L^p\,$-theory of  PDEs has advanced considerably in the last two or three decades due to improved techniques in modern harmonic analysis \cite{As2, BSV, GR, IM, IMpr, NTV} , stochastic processes \cite{Ba1, BH, Bu2, Bu1, DV, GMS, NTV} , quasiconvex calculus of variations \cite{Iw1, IL, PS} , complex interpolation \cite{AIPS} , etc.

It begins with the fundamental work by B. Bojarski \cite{Bo2,Bo3,Bo1} who inaugurated the $\,\mathscr L^p$-theory of the first order elliptic PDEs in the plane. 
He applied the  Calder\'{o}n-Zygmund type singular integral
\begin{equation}
 (\mathbf S \omega)(z) \;= \;- \frac{1}{\pi}  \iint_\mathbb C \frac{\omega(\xi)\, \textnormal d \xi}{ (z - \xi)^2 }\;,\;\;\; \omega \in \mathscr L^p(\mathbb C)
\end{equation}
which we refer  to as the \textit{Beurling Transform}, after its earliest appearance in A. Beurling's old lecture notes \cite{Beb1,Beb2}. Its significance to PDEs and Geometric Function Theory lies in the identity $ \mathbf S \circ \frac{\partial}{\partial \bar{z}} = \frac{\partial}{\partial z}$.
Higher dimensional $(n \geqslant 3)$ analogues of the Beurling Transform have been found in various contexts \cite{Iw1, IMpr,IMacta, IMb,BL} and the need to evaluate their $\mathscr L^p$-norms became evere more quintessential in the analytical foundation of multidimensional Geometric Function Theory.

Our primary aim is to further the interest in the $\,
\mathscr L^p\,$- norm of the Beurling Transform
\begin{equation}
S_p \bydef \| \, \mathbf S : \mathscr L^p (\mathbb C) \rightarrow \mathscr L^p (\mathbb C) \| \;,\;\;\; 1 < p < \infty
\end{equation}
The as yet unsolved conjecture \cite{Iw4} asserts that
\begin{conjecture}\label{p-normConjecture}
For all $\,1<p<\infty\,$ it holds
\begin{equation}
  S_p = p^* - 1 \bydef \;\left\{\begin{array}{ll} p-1\;,\quad\qquad\; \textnormal{if}\;\;\; 2\leqslant p < \infty\\
       \;\;1/(p-1) \;,\;\;\;\;\textnormal{if}\;\;\; 1<p\leqslant 2
       \end{array}\right.
\end{equation}
\end{conjecture}
This amounts to saying that
\begin{equation}\label{p-norm2}
  \Big{\|} \frac{\partial f}{\partial z}\,\Big{\|}_{\mathscr L^p(\mathbb C)}  \;\leqslant \; (p^* - 1) \,\Big{\|} \frac{\partial f}{\partial \bar{z}}\,\Big{\|}_{\mathscr L^p(\mathbb C)}\;,\;\;\;\;\textnormal{for}\;\;\; f \in \mathscr C^\infty _\circ(\mathbb C)
\end{equation}
or, equivalently
\begin{equation}\label{p-norm3}
  \iint _\mathbb C \Big{\{} \,|f_z(z)|^p \; - (p^*-1)^p |f_{\bar{z}}(z)|^p \Big{\}}  \;\textnormal d z   \leqslant \;0\;
  \;,\;\;\textnormal{for}\;\; f \in \mathscr C^\infty _\circ(\mathbb C)
\end{equation}

 Here the complex derivatives
$$\,\frac{\partial }{\partial \bar z }\, = \frac{1}{2} \left(\frac{\partial}{\partial x } \;+\; i \,\frac{\partial}{\partial y } \right ) \;\;\textnormal{and} \;\;\frac{\partial }{\partial z }\, = \frac{1}{2} \left(\frac{\partial}{\partial x } \;-\; i \,\frac{\partial}{\partial y } \right )  \,,\;\;z = x\, + \,i \,y\; $$
represent exactly two homotopy classes of the first order elliptic operators. These two classes are characterized  by the following topological property of the solutions to the corresponding homogeneous equations. In the class represented by $\,\partial /\partial \bar z\,$ the solutions are orientation preserving (with nonnegative Jacobian), whereas in the class of  $\,\partial /\partial z\,$ the solutions are orientation reversing. One of the strategic tasks for the theory of complex elliptic systems (linear and nonlinear)  is to establish precise $\,\mathscr L^p\,$-transition from $\,\frac{\partial f}{\partial \bar{z}}\,$ to  $\,\frac{\partial f}{\partial \bar{z}}\,$, which is the \textit{Beurling Transform}, .

 Thoughtful evidence to support Conjecture \ref{p-normConjecture} can be found in many articles. The interested reader is referred  to \cite{As2,AIMb,AIS,BM,Ball2,Ba1,BJ1,BJ2,BL,BW, BSV, DPV,DV,EH,GMS,Le,H,Iw4,Iw3,Iw5,Iw1,IKp,IMb,P,PW,NTV,VN} for numerous attempts, partial results and related topics.
 This elegant  
 mathematical problem has profound connections with the fundamental work of D.L. Burkholder on martingale inequalities and stochastic integrals \cite{Bu2,Bu1,Ba1,BJ1,BJ2,BL,BW,DV,GMS,L}, see  the extended survey article by R. Ba\~{n}uelos \cite{Ba1}. In fact the probabilistic study of the Beurling Transform was initiated in \cite{BW, L}, by  applying the  Burkholder integrals.
 Also some analogues of the Burkholder integrals have been found and developed for this purpose in dimensions $\,n\geqslant 2\,$, see \cite{Iw1,IMb,IL}. Today the studies of the Burkholder functions appear the most promising approach to Conjecture \ref{p-normConjecture}. The purpose of this note is to give a survey of the Burkholder functions from this point of view.	\\
\bigskip

\section{ A. Beurling, D. Burkholder and C.B. Morrey}  


 A continuous function  $\,
\mathbf E : \mathbb R^{\,m\times n} \rightarrow \mathbb R\,$, defined on the space of  $\,m\times n\, $-matrices, is said to be \textit{quasiconvex at} $\,A \in \mathbb R^{\,m\times n}\,$ if
\begin{equation}\label{quasiconvexity}
  \int_{\mathbb R^n} \left[\,\mathbf E(A + D\eta) \,-\,\mathbf E(A) \, \right]  \;\geqslant 0\,, \;\;\;\;  \textnormal{for every}\;\;\,\eta \in \mathscr C^\infty_0 (\mathbb R^{\,n}, \mathbb R^{\,m})\,.\end{equation}
 Here $ \, \eta  : \mathbb R^n \rightarrow  \mathbb R^m\,$ is a smooth mapping with compact support. We call $\mathbf E\,$  \textit{quasiconvex} if (\ref{quasiconvexity}) holds for all matrices $\,A \in \mathbb R^{\,m\times n}\,$.
 Quasiconvexity yields  convexity in the directions of rank-one matrices $\,X \in \mathbb R^{\,m\times n}$. Precisely, if $\mathbf E\,$ is quasiconvex,  then for every $\,A \in \mathbb R^{\,m\times n}\,$ the function of real variable t:
 \begin{equation} \;\;\;  \, t \mapsto \mathbf E(A+ t\,X)\,\textnormal{ is convex} \textnormal { whenever} \;\,rank\,X = 1\,.\end{equation}

 We refer to this later property of $\,\mathbf E\,$ as \textit{rank-one convexity}, see the seminal paper by C. B. Morrey \cite{M}.

 In general (in higher dimensions), the rank-one convexity does not imply quasiconvexity, see the famous example by V. \v{S}ver\'{a}k
\cite{Sv}.  C.B. Morrey himself was not quite definite in which direction he though things should be true \cite{M}. Nowadays, the case $\,m = n = 2\,$ remains an enigma for complex analysts \cite{BM, FS, Mu2, PS, Sv}. Our own thoughts in the spirit of Morrey's fundamental vision is the following.
\begin{conjecture}  \label{MorreyProblem}
The rank-one convex functions $\,\mathbf E : \mathbb R^{2\times 2} \rightarrow \mathbb R\,$ are quasiconvex.
\end{conjecture}
The dual concepts of \textit{quasiconcave} and \textit{rank-one concave} functions are formulated analogously: simply, we replace the word convex by concave. Equivalently, this amounts to considering $\, -\mathbf E\,$ instead of $\, \mathbf E\,$.
The most famous (and, arguably, the most important) example in two dimensions is the rank-one concave energy integral:
 \begin{equation}\label{BurkholderComplex}
\mathscr B^p_\Omega [f] \bydef  \int_\Omega \big[\,|f_z|\,-\, (p^{_\ast}-1) |f_{\bar z}| \, \big ] \cdot \big [\, |f_z| + |f_{\bar z}|\,\big ] ^{p-1} \, \textnormal d z \;,\;\;\; 1 < p < \infty
\end{equation}
  Such terms for the energy functionals pertain to all sorts of variational integrals whose integrands are rank-one concave functions (rank-one convex, quasiconvex, etc., respectively).
Here and in the sequel  we identify the gradient matrix $\,Df\,$ with the complex differential $\,df = f_{z}\,\textnormal{d}z \;+\; f_{\bar z}\,\textnormal{d}\bar{z} \,$\, or a pair of complex derivatives, whenever convenient. Accordingly,   $\, Df(z) \simeq \big(f_z , f_{\bar{z}}\big) \in \mathbb C \times \mathbb C\,\simeq \mathbb R^{2\times2}\,$. 

The special interest in the  function \eqref{BurkholderComplex} within the studies of the Beurling operator arises from the inequality 
\begin{equation} \label{BeBu}
C_p \cdot \big(\,|f_z|^p \; - (p^*-1)^{p} |f_{\bar{z}}|^{p} \bigr)  \leqslant \, \big(\,|f_{z}|\,-\,(p^*-1)\,|f_{\bar{z}}| \big)\cdot \big(\,|f_z|\;+\;|f_{\bar{z}}|\big)^{p-1} 
\end{equation}
which can be shown by elementary means, see e.g. \cite[Lemma 6.3.20]{stroock}. The positive constant $\,C_p = p\left(1 - \frac{1}{p^*}  \right)^{p-1}$ for $p>1$.
Thus in particular, Conjecture \ref{p-normConjecture} follows if one can prove that the Burkholder functions \eqref{BurkholderComplex} are quasiconcave at $A = 0$.

We shall work with the operator norm
 $$\, |Df(z)| = \max \{|Df(z)\, \textbf{v} | \,;\; |\textbf{v}| = 1\,\}\,=\, | f_z| + |f_{\bar{z}}|,$$
and the Jacobian determinant $$\,J_f(z) \bydef \det Df(z)  = |f_z|^2 - | f_{\bar{z}}|^2 \,.$$ In these terms the foregoing energy integral (\ref{BurkholderComplex}) can be expressed as:
\begin{equation}\label{BurkholderReal}
\mathscr B^p_\Omega [f] \;=\;   \frac{p^{_\ast}}{2}\int_\Omega \left[\, \det Df \, - \big| 1 - 2/p \,\big| \, |Df |^2 \, \,\right ] \cdot |\,Df\,| ^{p-2}\;\;.
\end{equation}

  That a pair of complex numbers $\, A = (\xi, \zeta) \in \mathbb C \times \mathbb C\,$ represents a rank-one matrix simply means that  $\,|\xi| = |\zeta| \neq 0\,$.
The nonlinear algebraic expression
\begin{equation}\label{BurkholderFunction}
\mathbf B_p\,(\xi, \zeta) \;\bydef\;  \big[\,|\,\xi|\,-  ( p^{_\ast}- 1)\, |\,\zeta| \, \,\big]\cdot \big[\, |\,\xi|\, + \,|\,\zeta|\,\big ] ^ {p-1}\,,
\end{equation}
(for vectors $\,\xi\,$ and $\,\zeta\,$ in any real or complex Hilbert space)  has emerged in Burkholder's theory of stochastic integrals and martingale inequalities \cite{Bu2, Bu1}.  He shows that the function $\, t \mapsto \,\mathbf B_p\,(\xi + t \,\alpha\,,\, \zeta + t\, \beta)\,$ of a real variable $\,t\,$ is concave whenever $\,|\alpha| \leqslant |\beta |\,$; in particular, if $\,|\alpha| = |\beta |\,$. Burkholder's computation, although planned for different purposes, when combined with (\ref{BurkholderComplex}) and (\ref{BurkholderReal})  reveals that $\,\mathscr B^p_\Omega [f]\,$ is rank-one-concave. It is this connection between Morrey's problem and Burkholder's work that inspired a search for the $\,n$\,-dimensional analogues of the rank-one-convex functionals suited to the $\,\mathscr L^p\,$-theory of quasiregular mappings \cite{Iw1}. Let us state it as:
\begin{theorem}
The matrix function $\,\mathbf E : \mathbb R^{n\times n} \rightarrow \mathbb R\,$, defined by
\begin{equation}\label{BurkholderHigherDimensional}
\mathbf E (A) \;\bydef\, \left[\; \pm\; \det A \,- \,\lambda \, |A|^n \,\right] \cdot  |A|^{p-n} \;,
\end{equation}
is rank-one-concave for all parameters $\,\lambda \geqslant |1 - \frac{n}{p} |\,$ and $\, p \geqslant \frac{n}{2}\,$. Moreover, $\,|1 - \frac{n}{p} |\,$ is the smallest value of $\,\lambda\,$ for which the rank-one-concavity holds.
\end{theorem}
\begin{definition}
We refer to (\ref{BurkholderFunction})  and its $\,n\,$-dimensional analogue (\ref{BurkholderHigherDimensional})  as \textit{Burkholder functions}. \\

Note that changing $\,\pm\,$ into $\,\mp\,$ in (\ref{BurkholderHigherDimensional}) results in the interchange of $\,|f_z|\,$ and $\,|f_{\bar{z}}|\,$ in (\ref{BurkholderFunction}). In particular, the rank-one concavity is unaffected. We confine ourselves to discussing the case of plus sign.
\end{definition}
\begin{conjecture}\label{ConjOnBurkholderFunctionals} Burkholder functions are quasiconcave.
\end{conjecture}
\noindent Further analysis of this and related conjectures see \cite{BM, BW, BL}.

Recently \cite{AIPS} , substantial progress has been made toward Conjecture \ref {ConjOnBurkholderFunctionals} in dimension $\,n = 2\,$.

\begin{theorem}\label{AIPS}
For $\, \frac{1}{K} = 1 - \frac{2}{p}\,$, the Burkholder energy $\, \mathscr B_p \,[f] \,,\, p \geqslant 2\,,$ is quasiconcave within $\, K\,$-quasiconformal extensions $\, f : \Omega \rightarrow \Omega \,$ of the identity boundary map. This just amounts to the following inequality
\begin{equation}\nonumber
\int_\Omega \mathbf B_p (Df)\; \textnormal{d} z \,\leqslant \,\int_\Omega \mathbf B_p (I)\; \textnormal{d} z\; = |\,\Omega |\;,
\end{equation}
whenever $\,f(z) \equiv z\,$ on $\,\partial \Omega\,$ and \;$\; \mathbf B_p (Df(z)) \geqslant 0\;$,  almost everywhere in $\, \Omega\,$.

\end{theorem}

 Far reaching novelties follow from this result.  Among the strong corollaries, we obtained weighted integral bounds for $\,K\,$-quasiregular mappings $\,f : \Omega \rightarrow \mathbb C\,$ at the borderline integrability exponent  $\, p = \frac{2K}{K-1}\,$,
 \\

\begin{center}
$
 \left[\,K - K(x)\,\right ] \, |Df(x)|^{\frac{2K}{K-1}} \; \in \mathscr L^1_{\textnormal{loc}} (\Omega)\;,\;\;\;\;\;\;\; K(x) \bydef \frac{|\, Df(x) \,|\!|^2 }{ \,\det\, Df(x)} \, \leqslant K \,.
$
\end{center}
\vskip0.1cm

\noindent These sharpen and generalise  the optimal higher integrability bounds for quasiconformal mappings proven in \cite{As1, As2}

Among further consequences of Theorem \ref{AIPS} we find that quite general classes of radial maps are local maxima for $\mathscr B^p_\Omega [f] $. These facts will be elaborated in more detail in Section \ref{localmax}.

$\,K\,$-quasiconformal extensions $\, f : \Omega \rightarrow \Omega \,$ of the  identity boundary map $\,\textnormal{Id} :  \partial \Omega \rightarrow \partial \Omega \,$, with maximal  Burkholder energy, have been presented in \cite{AIPS}.  It became reasonable to speculate that Theorem  \ref{AIPS} presents Conjecture \ref{ConjOnBurkholderFunctionals} in its worst-scenario. The novelty of our approach lies in using an analytic family of the Beltrami equations, which manifests the intricate nature of Conjecture \ref{ConjOnBurkholderFunctionals}.  \\

\section{Enquiry on quasiconvexity at $\, 0 \in \mathbb R^{n\times n}\,$}
In spite of the example by V. \v{S}ver\'{a}k \cite{Sv}, which answers the general question of quasiconcavity of rank-one concave functions  in the negative, it is still reasonable to inquire about quasiconcavity at $\, A = 0 \in \mathbb R^{n\times n}\,$. Let us take a quick look at the integrands $\,\mathbf E : \mathbb R^{n\times n} \rightarrow \mathbb R\,$ which are $\,p\,$-homogeneous at infinity; that is,
$$
\mathbf E(t A) = t^p \,\mathbf E(A) + o (t^p)\;,\;\textnormal{uniformly as}\; |A| \leqslant \textnormal{constant} \; \;\;\textnormal{and}\; t \rightarrow \infty\,.
$$
Suppose $\,\mathbf E\,$ is quasiconcave at some  $\, A \in \mathbb R^{n\times n}\,$. It is not difficult to see that $\,\mathbf E\,$ is automatically quasiconcave at  $\, 0 \in \mathbb R^{n\times n}\,$. The converse is far from being true. This can easily be seen in case of the   \textit{Beurling} energy,
\begin{equation}
\mathscr F_\Omega^{p , M} [f] \,\bydef  \, \iint_\Omega \big[\, |f_z|^p -  M ^p |f_{\bar{z}}|^p \; \big] \;,\;\;\;\;\; f \in \mathscr W^{1,p}_\circ (\Omega)\;,\; p> 1
\end{equation}
where $\, M  \geqslant S_p\,$-the $\,\mathscr L^p\,$-norm of the Beurling transform. By the very definition of $\,S_p\,$ it follows that $\,\mathscr F^{p, M}_\Omega[f] \leqslant 0 \,$, for  $f \in \mathscr W^{1,p}_\circ (\Omega)\,$. In other words, the \textit{Beurling function}
\begin{equation} \label{BeurlingFunctional}
\,\mathbf F_p^M (\xi, \zeta) \,\bydef \,  |\,\xi|^p -  M ^p |\,\zeta|^p \;,\;\; M \geqslant S_p\;,\; (\xi, \zeta) \in \mathbb C \times \mathbb C
\end{equation}
is quasiconcave at the origin. On the other hand, when $\,p \neq \,2\,$,   $\,\mathbf F_p^M\,$ is not quasiconcave (even for $\, M > 0\,$). In fact $\,\mathbf F_p^M\,$ fails to be rank-one-concave. For this, examine the function $\, t \mapsto\,|\,\xi +t|^p -  M ^p |\,\zeta + t |^p \,$  for concavity at $\,t \approx 0\,$. When $\, p > 2\,$  concavity fails if  $\,(\xi, \zeta) \approx (1, 0) \,$ .  When $\, 1<p <2 \,$ concavity fails if  $\,(\xi, \zeta) \approx (0, 1)\,$. \\
It is therefore more realistic to insist that 
\begin{conjecture}\label{Quasiconvexity-at-0}
Burkholder functions are quasiconcave at the zero matrix.
\end{conjecture}
\noindent  which is still sufficient for Conjecture \ref{p-normConjecture}.
  Further, an affirmative answer  would give us optimal $\,\mathscr L^p\,$-estimates of the gradient of $\,n\,$-dimensional quasiconformal mappings and the associated nonlinear PDEs. Up to now, quasiconcavity at zero for the functional (\ref{BurkholderHigherDimensional}) has been established for $\, \lambda = \lambda_p (n)\, < 1$  sufficiently close to $1$,    with $\,p \geqslant n - \varepsilon\,$ for some small $\,\varepsilon > 0\,$  \cite{Iw5, Iw1}. At this point it is constructive to introduce an additional parameter to Burkholder integrand.
\begin{equation}\label{BurkholderM}
\mathbf B^M_p(\xi,\zeta) \bydef  \big[\,\,|\xi| -  M\, |\zeta|  \,\big ] \cdot \big [\, |\xi| + |\zeta|\,\big ] ^{p-1} \;
\end{equation}
The rank-one-concavity still holds if $\,M\geqslant p^{_\ast}-1\,$. The $\,\mathscr L^p\,$-boundedness of the Beurling transform  $\mathbf S$ implies that if $\,M\,$ is sufficiently large, then $\,\mathscr B^{p,\,M}_\Omega [f]\,$ is quasiconcave at zero. It should, therefore,  come as surprise that
\begin{remark}\label{QuasiconvexityForMlarge} Quasiconcavity of
$\mathscr B^{p,\,M}_\Omega \,$ remains unknown for any $\,M\geqslant p^{_\ast}-1\, .$
\end{remark}

Also note that we have the following point-wise inequality
\begin{equation}
\; |f_z|^p - \;M^p |f_{\bar{z}} | ^p \,\;\leqslant\; p\,\left(\frac{M}{1 + M}\right)^{p-1}  \big[\,|f_z|  -  M\, |f_{\bar z}|\,\big ] \cdot \big [\, |f_z| + |f_{\bar z}|\,\big ] ^{p-1}
\end{equation}
whenever $\,M \geqslant p^{\,\ast} - 1\,$, see Lemma 8.1 in \cite {Iw1}.
\begin{example}  By way of digression,  consider the following rank-one concave function,
\begin{equation}\label{G.AubertProc.Roy.Soc.Edin.106,1987}
  \mathbf A(\xi, \zeta)\,=\, \big[\, |\xi|^2 \,-\,M^2 \,|\zeta|^2 \,\big ] \cdot \big [\, |\xi|^2 + |\zeta|^2\,\big ]\;,\;\;\,M \geqslant 2 + \sqrt{3}\,
\end{equation}
For the original source of this function we refer the reader to \cite{Aubert}.
The lower bound  $\,M \geqslant 2 + \sqrt{3}\,$ is the best possible for the rank-one concavity of $\,\mathbf A(\xi, \zeta)\,$. It is not difficult to see that for every $\, M \geqslant 1\,$, there is a unique constant $\, c > 0\,$ such that
\begin{equation}\label{Quasiburkholder}
 |f_{z} | ^4 \,-\,M^4 \, |f_{\bar{z}}|^4\;\leqslant  c\, \big[\,  |f_{ z}|^2 \,-\,M^2\,|f_{\bar{z}}|^2 \,\big ] \cdot \big [\, |f_z|^2 + |f_{\bar z}|^2\,\big ]
\end{equation}
Actually, given the factor $\,M^2\,$ in the right hand side,  the inequality (\ref{Quasiburkholder}) forces  $\,c\,$ to be equal to $\,\frac{2\,M^2}{1+ M^2}\,$.
Never mind, even in the best scenario (conjectural quasiconcavity for  $\, M = 2+\sqrt{3} > 3 \,$), the approach by using $\,\mathbf A(\xi\,,\zeta)\,$  would not result in the exact value of the $\,\mathscr L^{\,4}\,$-norm of the Beurling transform. Thus there is no prospect of gaining any good $\,\mathscr L^{\,4}\,$-estimates through the  rank-one concavity of $\,\mathbf A(\xi, \zeta)\,$ and the inequality (\ref{Quasiburkholder}). For more examples of rank-one functions we refer the reader to \cite{Aubert, Sv2}.
\end{example}

\section{rank-one concave envelopes}

\begin{definition} Given a continuous function  $\,
\mathbf E : \mathbb R^{\,m\times n} \rightarrow \mathbb R\,$, we use a visual notation to define:
\begin{itemize}
\item \textit{Rank-one concave envelope} of  $\,\mathbf E\,$ (\textit{the smallest majorant}) as,
    $$
    \mathbf E_R^{\smallfrown}  =  \; \inf\{\,\Xi \,; \;\, \Xi :\, R^{\,m\times n} \rightarrow \mathbb R\,\textnormal{ is rank-one concave, and}\;\Xi \geqslant  \mathbf E  \}
    $$
    \item \textit{Quasiconcave envelope} of  $\,\mathbf E\,$ as,
    $$
    \mathbf E^{\smallfrown}_{Q}\, =\; \inf\{\,\Xi \,; \;\, \Xi :\, R^{\,m\times n} \rightarrow \mathbb R\,\textnormal{ is quasiconcave, and}\;\Xi \geqslant  \mathbf E  \}
    $$
\end{itemize}
\end{definition}

Obviously $\,\mathbf E^{\smallfrown}_{Q} \geqslant \mathbf E_R^{\smallfrown}\,$ pointwise; the former function being quasiconcave   and the latter rank-one concave.
\begin{theorem}
 Recall the Beurling function $\,\mathbf F_{\!p} : \mathbb C\times \mathbb C  \rightarrow \mathbb R\,$
$$
\mathbf F_{\!p} \,(\xi, \zeta) \,\bydef \,  |\,\xi |^p -  \left(p^{\,\ast} - 1 \right) ^p |\,\zeta|^p \;,\;\; \;\;\;\;1 < p < \infty\,.
$$
and the Burkholder's function
$$\mathbf B_p\,(\xi, \zeta) \;\bydef\;  \big[\, |\,\xi| -  (p^{_\ast}- 1)\, |\,\zeta|  \,\big]\cdot \big[\, |\,\xi|\, + \,|\,\zeta|\,\big ] ^ {p-1}\,
$$
The rank-one concave envelope of $\mathbf F_p$ is given by the following formula. For $p \geqslant 2$, 

\begin{displaymath}
 \mathbf F^{\smallfrown}_p\, (\xi, \zeta) \; = \; \left\{\begin{array}{ll}
 |\,\xi |^p  - (p^{\,\ast} - 1) ^p |\zeta|^p \; \;=\; \mathbf F_{\!p}\,(\xi, \zeta)\;\; & \textrm{if $\;\;\;\left( p^{\,\ast} - 1 \right) |\zeta| \;\geqslant \; |\xi |      \,$}  \\
    p\left( 1 - 1/p^{\,\ast}\right )^{p-1}\,\mathbf B_p & \textrm{if $\;\;\;\left( p^{\,\ast} - 1 \right) |\zeta| \;\leqslant \; |\xi |      \,$}
\end{array} \right.
\end{displaymath}
While, for $1<p<2$, 
\begin{displaymath}
 \mathbf F^{\smallfrown}_p\, (\xi, \zeta) \; = \; \left\{\begin{array}{ll}
 p\left( 1 - 1/p^{\,\ast}\right )^{p-1}\,\mathbf B_p  & \textrm{if $\;\;\;\left( p^{\,\ast} - 1 \right) |\zeta| \;\geqslant \; |\xi |      \,$}  \\
   \mathbf F_{\!p}\,(\xi, \zeta)\;\;  & \textrm{if $\;\;\;\left( p^{\,\ast} - 1 \right) |\zeta| \;\leqslant \; |\xi |      \,$}
\end{array} \right.
\end{displaymath}
\end{theorem}

Burkholder \cite{Bu1} shows this in a slightly different sense. Namely, that the envelope function above is the smallest majorant of $\mathbf F_p$ which is concave in orientation-reversing directions (as discussed on page 5). See also, p. 64 in  \cite{Ba1}. The result as stated here basically follows from the work \cite{VV}.

\begin{proof}
Let us denote by $\mathbf{E}(\xi,\zeta)$ the formula given above. Our task is to show that $\mathbf F^{\smallfrown}_p=\mathbf{E}$. For any pair $\theta_1,\theta_2 \in [0,\pi)$, consider the function $\mathbf F^{\smallfrown}_{p,\theta_1,\theta_2} \colon \mathbb{R} \times \mathbb{R} \to \mathbb{R}$,
\[ (x,y) \mapsto \mathbf F^{\smallfrown}_p(e^{i\theta_1}x,e^{i \theta_2}y).
\]
Using rank-one concavity of $\mathbf F^{\smallfrown}_p$ we see that $\mathbf F^{\smallfrown}_{p,\theta_1,\theta_2}$ is zig-zag concave, that is, concave in the directions of $\pm\, \pi/4$ in $\mathbb{R}^2$.
By the results (Theorem 6 and 7) of \cite{VV} on the zig-zag concave envelope of $|x|^p-(p^*-1)^p |y|^p$, we have that $\mathbf F^{\smallfrown}_{p,\theta_1,\theta_2}(x,y) \geqslant \mathbf{E}(|x|,|y|)$. Since, this is true for any $\theta_1,\theta_2 \in [0,\pi)$ we have the inequality $\mathbf F^{\smallfrown}_p(\xi,\zeta) \geqslant \mathbf{E}(|\xi|,|\zeta|)=\mathbf{E}(\xi,\zeta)$. On the other hand, as we have remarked $\mathbf{E}$ is rank-one concave so $\mathbf F^{\smallfrown}_p=\mathbf{E}$ as claimed.
\end{proof}

\section{Radially linear transformations} It is advantageous to dispose with a fairly large class of mappings that can be effectively applied to all rank one-concave functionals when computing the energy. One of such classes  is the following:\\
Suppose we are given a Lipschitz function $\, \Lambda : [0, R] \rightarrow \mathbb R^{n\times n}\,$. Define a mapping $\, f \,: \mathbb B_R \rightarrow \mathbb R^n\,,\; \mathbb B_r = \{ x \,; \, |x| \leqslant r\,\}\,$,  by the rule
\begin{equation}\label{TwistedRadialMap}
 f(x) = \Lambda(|x|) \,x
\end{equation}
Thus $\,f\,$ restricted to any sphere $\,\mathbb S_r = \{ x ; \; |x| = r\,\}\,,\, 0 < r \leqslant R\,$, is a linear transformation. 
Proceeding further in this direction one could obtain more mappings of interest, but for us the class of mappings defined by (\ref{TwistedRadialMap}) will work perfectly well. For radial maps where $\, \Lambda : [0, R] \rightarrow \mathbb R$, the following proposition is shown e.g. in \cite[Proposition 3.4]{Ball90}.

\begin{proposition} \label{EnergyRdialLinear} Let $\,\textnormal E : \mathbb R^{n\times n} \rightarrow \mathbb R\,$ be continuous and rank-one concave. Then for $\,f(x) = \Lambda(|x|) \,x \,$ as in (\ref{TwistedRadialMap}), we have
$$
\mathscr E [f] \,\bydef \, \int_{\mathbb B_R} \textnormal E( Df ) \leqslant \int_{\mathbb B_R} \textnormal E( \Lambda(R)) \,=\, \mathscr E[f^R]\;,\;\textnormal{where}\;\; f^R(x)  \bydef \Lambda(R)\,x
$$
\end{proposition}
\begin{proof}
 A standard mollification procedure, through convolution of $\,\textnormal E\,$ with an approximation of the Dirac mass,
$$
\textnormal E_\varepsilon (X)\, \bydef \,(\textnormal E \ast \Phi_\varepsilon) (X) \; = \int_{\mathbb R^{n\times n}}  \Phi_\varepsilon(Y) \, \textnormal E(X - Y ) \,\textnormal{d} Y\;,
$$
results in $\,\mathscr C^{\infty}\,$-smooth functions which are still rank-one concave. As $\,\varepsilon\,$ approaches  0 the mollified functions $\,\textnormal E_\varepsilon\,$  converge to $\,\textnormal E\,$  uniformly on compact subsets of $\,n\times n\,$- matrices. Therefore, there is no loss of generality in assuming that $\,\textnormal E \in \mathscr C^\infty( \mathbb R^{n\times n} \,, \mathbb R)\,$. With this assumption
consider the linear mappings $\,f_t = \Lambda(t)\,x\,$, for $\,0 \leqslant t \leqslant R\,$. We aim to show that the difference of energies:
\begin{eqnarray}\nonumber
 \mathscr E(t) \;&\bydef &\, \int_{|x| \leqslant \,t} \textnormal E(Df(x))\,\textnormal d x \;-\; \int_{|x| \leqslant \,t} \textnormal E(Df_t(x))\,\textnormal d x\;\\ & = &\int_{|x| \leqslant \,t} \textnormal E(Df(x))\,\textnormal d x \;-\;\frac{\omega_{n-1}}{n}\,t^n \textnormal E\left(\Lambda(t) \right)\nonumber
\end{eqnarray}
 is nondecreasing in $\,t\,$. Thus we compute its derivative for $\,t\geqslant 0\,$. The computation is legitimate at almost every $\,t \in [0 , R]\,$,
$$
\mathscr E\,'(t)\,=\int_{|x| = \,t} \textnormal E(Df(x))\,\textnormal d x \;-\; \omega_{n-1}\,t^{n-1} \textnormal E\left(\Lambda(t) \right) \;-\;
\frac{\omega_{n-1}}{n}\,t^n \big \langle\textnormal E '\left(\Lambda(t)\right) \,| \,\Lambda ' (t) \big\rangle
$$

Next we find that $\,Df(x) = \Lambda (|x| ) \, + \Lambda ' (|x| )\, \frac{x \otimes x}{|x|}\,$, where the tensor product of vectors represents a rank-one matrix. By virtue of rank-one concavity of $\,\textnormal E\,$ it follows that
$$\,\textnormal E(Df(x)) \leqslant \, \textnormal E \left (\Lambda(|x|) \right )\; +\; {\Big \langle} \textnormal E'\big (\Lambda(|x|)\big )\, \Big|\; \Lambda ' (|x| )\, \frac{x \otimes x}{|x|}\, \Big \rangle\,.$$
We then integrate over the sphere $\, |x| = t\,$, to obtain

\begin{eqnarray}\nonumber
 \int_{|x| = \,t} \textnormal E(Df(x))\,\textnormal d x \;&\leqslant &  \; \omega_{n-1}\,t^{n-1} \textnormal E\left(\Lambda(t) \right) \; + \; {\Big \langle} \textnormal E'\big (\Lambda(t)\big )\, \Big|\; \frac{\Lambda ' (t)}{t} \int_{|x| = t} x \otimes x\, \Big \rangle\,,\\ &\textnormal{where} & \int_{|x| = t} x \otimes x\, = \frac{1}{n} \int_{|x| = t} |x|^2  \, I = \,\frac{\omega_{n-1}}{n} \,t^{n+1} \, I\nonumber
\end{eqnarray}
In conclusion, $\,\mathscr E\, '(t) \leqslant 0\,$ almost everywhere. Hence $\,\int_{\mathbb B_R} \textnormal E( Df ) \leqslant \int_{\mathbb B_R} \textnormal E( \Lambda(R))\,$, as desired.

\section{Burkholder's energy of radial stretchings}\label{EnergyRadial} Of particular interest are mappings, subject to the given boundary data,  at which the Burkholder energy  assumes the maximum value. For this we look at the radial stretchings as in \cite{AIPS, BM}. Our notations, however,  are little different. Let
\begin{equation}\label{RadialStretchings}
f_{_+}(z) = \rho(|z|) \frac{z}{|z|}\;\;\;\textnormal{and}\;\;\;\; f_{_-}(z)\, \bydef \,\rho(|z|) \,\frac{\bar{z}}{|z|}\;.
\end{equation}
Here the continuous function $\,\rho: [0 , R ] \rightarrow [0, \infty)\,$ is assumed to be locally Lipschitz in $(0,R]$  and  satisfy $\,\rho(0) = 0\,$. However, we do not require that $\rho$ is increasing, in particular $f_\pm$ needs not to be a homeomorphism. In our situation $\, f_{_+}\,$ and $\,f_{_-}\,$ have well defined complex derivatives for almost every $\,z \in \mathbb D_R = \{ z ;\; |z| \leqslant R\,\}\,$,
$$
 \frac{\partial f_{_+}}{\partial z} (z) = \frac{1}{2}\left[\frac{\rho(|z|)}{|z|} \,+\, \rho\,'(|z|) \right]\;,\;\;\; \frac{\partial f_{_+}}{\partial \bar{z}} (z) \,= \frac{1}{2}\left[\rho\,'(|z|) - \frac{\rho(|z|)}{|z|} \right] \frac{z}{\bar{z}}
$$
$$
\frac{\partial {f_{_-}}}{\partial \bar{z}} (z) = \frac{1}{2}\left[\frac{\rho(|z|)}{|z|} \,+\, \rho\,'(|z|) \right]\;,\;\;\; \frac{\partial {f_{_-}}}{\partial z} (z) \,= \frac{1}{2}\left[\rho\,'(|z|) - \frac{\rho(|z|)}{|z|} \right] \frac{\bar{z}}{z}
$$
In addition to $\,\rho\,$ being Lipschitz, we wish that $\, |Df| = |f_z| + |f_{\bar{z}}|\, $ be free from the derivative of $\,\rho\,$. This is equivalent to requiring that
\begin{equation}\label{RhoCondition1}
- \rho(r) \leqslant r \,\rho\,' (r)\;\leqslant \rho(r)\;,\;\;\;\;\textnormal{for almost every }\;\; r \bydef |z| \leqslant R
\end{equation}
Finally, in case $p>2$ we also assume that
\begin{equation}\label{RhoCondition2}
\lim_{r\to 0^+}r^{-1+2/p}\rho (r)=0.
\end{equation}
Thus
$$
 \left |\frac{\partial f_{_\pm}}{\partial z} (z)\right | = \frac{1}{2}\left(\frac{\rho(|z|)}{|z|} \,\pm\, \rho\,'(|z|) \right)\;\;\textnormal{and}\;\;\; \left |\frac{\partial f_{_\pm}}{\partial \bar{z}} (z)\right | \,= \frac{1}{2}\left( \frac{\rho(|z|)}{|z|} \mp \rho\,'(|z|) \right)
$$
In either case $\, |Df(z)| = \rho(|z|)/|z|\,$.  The $\,\mathbf B_p\,$-energy of $\, f\,$ can then be computed; we take for $\,f\,$ the radial stretching $\,f_{_+}\,$ if $\,2 \leqslant p < \infty,$ and $\,f_{_-}\,$ if $\,1 <p  \leqslant 2 \,$.
\begin{eqnarray}
&&\mathscr B_p \,[f_{_\pm}] \;\bydef \int_{|z|\leqslant R} \mathbf B_p(f_z, f_{\bar{z}})\;=\;\textnormal{$ \hskip9cm$}\nonumber\\ &=& \;\int_{|z|\leqslant R} \big[\,|f_z|\, - \, (p^{_\ast}-1) |f_{\bar z}|  \,\big ] \cdot \big [\, |f_z| + |f_{\bar z}|\,\big ] ^{p-1} \, \textnormal d z \;\nonumber\\&=&
\frac{1}{2}\int_{|z|\leqslant R} \left[\, \Big(2 - p^{\ast}\Big) \frac{\rho(|z|)}{|z|} \,\pm\,p^{\ast} \,\rho\,'(|z|) \,\right] \cdot \left [\,\frac{\rho(|z|)}{|z|}\,\right ] ^{p-1} \, \textnormal d z \nonumber\\&=&
 \pi \int_0^R \big [(2 - p^{\ast})\, r^{1-p} \rho^p \;\pm\; p^{\ast} \,r^{2-p} \rho^{p-1} \rho\,' \big ] \textnormal d r\nonumber\\&=&
\pm \frac{\pi p^{\ast}}{p} \int_0^R \frac{\textnormal{d}}{\textnormal{d}r} \left( r^{2-p} \rho ^p\right )\;\textnormal{d} r \;=\; \pm \frac{\pi p^{\ast}}{p} R^{2-p} \big[\rho(R) \big]^p\;= \mathscr B_p \,[f^R_{_\pm}]. \nonumber
\end{eqnarray}

This is none other than the $\,\mathbf B_p\,$-energy of the linear extension of the boundary map $\,f_{_\pm} : \mathbb S_R \rightarrow \mathbb C\,$; that is, $\,f^R_+(z) = \frac{\rho(R)}{R}\,z\; (2\leqslant p < \infty)\,$ and  $\,f^R_-(z) = \frac{\rho(R)}{R}\,\bar{z}\; (1 < p \leqslant 2 )\,$.
\end{proof}

\section{
Burkholder function is an extreme point} \label{ExtremePoint}

Let $\,\mathscr V\,$ be a real vector space and $\,\mathcal F \subset \mathscr V\,$ a convex subset. An extreme point of $\mathcal F\,$ is an element $\,F \in \mathcal F\,$ which does not lie in any open segment joining two elements of $\,\mathcal F\,$.\\
We shall consider the vector space $\,\mathscr V = \mathscr V_p \,$ of continuous functions $\,\mathbf E : \mathbb C \times \mathbb C \rightarrow \mathbb R\,$ which are isotropic and homogeneous of degree $\, 1 < p < \infty \,$. Precisely,
\begin{itemize}
\item we assume that$\,\mathbf E(\xi, \zeta)= \Phi(|\xi|\,, |\zeta| )\,$ for some locally Lipschitz function $\, \Phi : [0, \infty) \times [0, \infty) \rightarrow \mathbb R\,$ , and
\item
$
 \mathbf E(t \xi, \,t\zeta)\; =\; t^p \,\mathbf E(\xi, \zeta)\;,\;\; \textnormal{for}\; t \geqslant 0\,\,\;\textnormal{and}\; \;\xi, \zeta \in \mathbb C\;.
$
\end{itemize}
 Recall that $\,\mathbf E \in \mathscr V\,$ is rank-one convex (concave) if for every $\,\xi, \zeta \in \mathbb C\,$ and $\, \xi_\circ , \zeta_\circ \in \mathbb S^1\,$ \, \,the real variable function  $\, t \mapsto \mathbf E(\xi +\,t\xi_\circ , \; \zeta + \,t\zeta_\circ ) \, $ is convex (concave, respectively).

\begin{definition} We let $\,\mathscr V_p^{\,\smallsmile}   \subset \mathscr V_p\,$  and $\,\mathscr V_p ^{\,\smallfrown}  \subset \mathscr V_p\,$ denote the families of rank-one convex and rank-one concave functions, respectively.
\end{definition}

Both families $\,\mathscr V_p^{\,\smallsmile}\,$ and $\,\mathscr V_p ^{\smallfrown}\,$ are convex subsets of $\,\mathscr V_p\,$.


Before proceeding to the extreme points we need to look at a slightly more general context. Suppose we are given a decomposition of the Burkholder function $\,\mathbf B_p = \mathbf B^\smallfrown_p \in \mathscr V_p{\,^\smallfrown} \,$ (and similarly $\,-\mathbf B_p\in \mathscr V_p^{\,\smallsmile}\,$).

\begin{equation}\label{Decomposition1}
\mathbf B_p (\xi ,\zeta) =\sum_{1\leqslant i \leqslant n}   \lambda_{\,i}\,\mathbf E_{\,i}(\xi , \zeta) \;,\;\; \lambda_{\,i} > 0\,,\;\;\textnormal{where $\,\mathbf E_{\,i} \in \mathscr V_p ^\smallfrown\,$.}
\end{equation}
One possibility is that there exist positive numbers $\,\theta _i > 0\,$ such that
\begin{equation}\label{DecompositionOutcome1}
\mathbf E _i \;\equiv\; \theta _i \,\mathbf B_p\;,\;\;\textnormal{for all} \; i = 1, 2, ... , n\;,\;\;\textnormal{and}\;\; \sum_{1\leqslant i \leqslant n}   \lambda_i\,\theta_i  \;= 1 .
\end{equation}
\begin{proposition}\label{ResultOfDecomposition}
 For $\,p \neq 2\,$, a decomposition of Burkholder function $\,\mathbf B_p\,$ as in (\ref{Decomposition1}) forces its components $\,\mathbf E _i\,$ to satisfy (\ref{DecompositionOutcome1}). For $\,p=2\,$, however,  the Burkholder function is a null-Lagrangian (i.e. it is both quasiconcave and quasiconvex), $\,\mathbf B_2 (\xi, \zeta ) = |\zeta|^2 - |\xi |^2\,$. In this case each component $\,\mathbf E_i(\xi,\zeta)\,$ is a real (positive or negative) multiple of $\,\mathbf B_2 (\xi, \zeta )\,$.
\end{proposition}
The key observation to the proof is that Burkholder energy $\,\mathscr B_p [f] \,$  admits many stationary solutions. Among those are a number of radial power stretchings.
\begin{proof}
Let us test (\ref{Decomposition1}) with the radial stretchings as in \eqref{RadialStretchings}, $\,f = f_+\,$ if $ \,2 \leqslant p < \infty\,$ and $\,f = f_-\,$ if $ \,1 < p \leqslant 2\,$, requiring that \eqref{RhoCondition1} holds and additionally that $\rho(t) = t$ for $0 \leqslant t \leqslant 1$. Computing their energies in the disc $B(0,R)$
we have 
\begin{eqnarray}  &&\mathscr B_p [f] = \sum_{1\leqslant i \leqslant n}   \lambda_i\,\mathscr E_i [f] \; \leqslant \sum_{1\leqslant i \leqslant n}   \lambda_i\,\mathscr E_i [f^R] \, \;\;\;\;\;\Big \|\;\;{\textnormal{because} \;\; \mathscr E_i [f] \; \leqslant \mathscr E_i [f^R]\atop \textnormal{for every} \; i = 1, 2,..., n }\, \nonumber \\ &=& \sum_{1\leqslant i \leqslant n}   \lambda_i\,\pi R^{2-p} \,[\,\rho(R)\,]^p \,\mathbf E_i (I_\pm) \;\;\;\;\;\;\;\;\;\;\;\;\;\Big \|\,\; \textnormal{where}\; \;{\;\,I_+ = \textnormal{id}  \;\;,\, 2 \leqslant p <\infty\,\;\atop \;\,I_- = \overline{\textnormal{id}} \;\; ,\,1<p\,\leqslant \;2\,\,} \nonumber \\ & =&      \sum_{1\leqslant i \leqslant n}   \lambda_i\,\pi R^{2-p}\, [\,\rho(R)\,]^p \,\mathbf B_p (I_\pm) \; \theta _i \;\;\;\;\;\;\;\; \,\;\;\;\;\;\;\;\;\;\;\;\;\;\;\;\;\,\Big \|\;\,\textnormal{where}\;\;\theta _i = \frac{\mathbf E_i (I_\pm)}{\mathbf B_p (I_\pm)} \, \nonumber \\&=& \sum_{1\leqslant i \leqslant n}  \lambda_i\,\mathscr B_p [f] \; \theta_i \; 
 \nonumber
\end{eqnarray}
As obviously $\sum_{1\leq i\leq n} \lambda_1\theta_1=1,$ we see that this chain is possible only if
$$\mathscr E _i [f] \;=\; \theta _i \,\mathscr B_p [f] =\; \theta _i \,\mathscr B_p [f^R]\;,\;\;\textnormal{for all} \; i = 1, 2, ... , n\;, \;\textnormal {and all} \; R \geqslant 1\,. $$

We write it as:
$$\, \int_{|z|\leqslant R }\mathbf E _i [Df]\,\textnormal{d} z \;=\; \theta _i \,\int_{|z|\leqslant R }\mathbf B_p [Df]\,\textnormal{d} z .$$

Note that $\,Df (z) \equiv I_\pm \,$ for $ |z| \leqslant 1 \,$ and, by the definition of $\,\theta_i\,$, $\,\mathbf E _{\,i}\, [I_\pm] = \theta _{\,i} \mathbf B_p \,[I_\pm] \,$. Hence $\, \int_{|z|\leqslant 1 }\mathbf E _i [Df]\,\textnormal{d} z \;=\; \theta _i \,\int_{|z|\leqslant 1}\mathbf B_p [Df]\,\textnormal{d} z .$ The energy  equation reduces to:

\begin{equation}\label{ReducedIdentity}
\, \int_{1\leqslant|z|\leqslant R }\mathbf E _i [Df]\,\textnormal{d} z \;=\; \theta _i \,\int_{1 \leqslant |z|\leqslant R }\mathbf B_p [Df]\,\textnormal{d} z .
\end{equation}
We test this  by further specifying the radial stretchings also in the annulus $1\leq |z|\leq R$  by setting
\begin{equation}\label{Testfunctions}
f(z) = f_+(z)  =  |z|^{\alpha - 1}\;z  \;,\;\;\;\; f(z) = f_-(z)  =  |z|^{-\alpha - 1}\;\bar{z}, \quad \; -1 \leqslant \alpha \leqslant 1.
\end{equation}
Then $f$ is quasiconformal in the annulus if $\alpha\not=0$, but one might observe that 
$f$ is a homeomorphism of $\{ |z|<R\}$ only if $\alpha >0.$ 
In any case
\begin{equation}
2 \,|f_z(z)|  =  \,(\alpha + 1)  \,|z|^{\alpha - 1}  \;,\;\;\;\; 2 \,|f_{\bar{z}}(z)| = ( -\,\alpha + 1) \,|z|^{\alpha - 1}
\end{equation}
Substitute these formulas into (\ref{ReducedIdentity}) to obtain
\begin{equation}\nonumber
\, \int_{1\leqslant|z|\leqslant R } |z|^{\alpha p - p }\,\mathbf E _{\,i} (\alpha + 1  \,, -\alpha + 1 )\, \;=\; \theta _{\,i} \,\int_{1\leqslant|z|\leqslant R } |z|^{\alpha p - p }\,\mathbf B_p (\alpha + 1 \, , - \alpha + 1 )\,.
\end{equation}
Hence
\begin{equation}\nonumber
\mathbf E _{\,i} \big(\alpha + 1\, , - \alpha +1\, \big)\, \;=\; \theta _{i} \,\mathbf B_p \big(\alpha + 1  \,,\, - \alpha + 1\, \big)\,.
\end{equation}
By homogeneity and isotropy,
\begin{equation}\nonumber
\mathbf E _{\,i} (\xi\, , \,\zeta )\, \;=\; \theta _i \,\mathbf B_p (\xi \,,\, \zeta )\,,\;\;\;\textnormal{for all}\;\; \xi , \zeta \in \mathbb C
\end{equation}
Now, for $\,p\neq 2\,$ , since both $\,\mathbf E _{\,i}\,$ and $\, \mathbf B_p\,$ are of the same rank-one convexity type, we conclude that $\,\theta_i > 0\,$. However, in case $\,p= 2\,$ (null-Lagrangians) the coefficients  $\,\theta_{i} > 0\,$ are allowed to be negative as well.
This completes the proof of Proposition \ref{ResultOfDecomposition}.
\end{proof}

The proof of Proposition  \ref{ResultOfDecomposition} has an interesting consequence.

\begin{corollary} Let $\,\mathbf E \in \mathscr V_p \,$. Regardless of whether $\,\mathbf E\,$  is rank-one concave or not, the identity
\begin{equation}
\mathscr E[f] \bydef \int_{|z| \leqslant 1} \mathbf E(|f_z|\,,|f_{\bar{z}}| \,)\, \textnormal{d}z = \mathscr E[\textnormal{Id}] = \pi\;, \end{equation}
for all $ \; f(z) = \rho(|z|) \frac{z}{|z|}\,$ as in (\ref{RhoCondition1})\,,
yields  $\,\mathbf E(\xi , \,\zeta) = \mathbf B_p (\xi, \zeta)\,$. In particular, $\,\mathbf E\,$ must be rank-one concave.
\end{corollary}

We now introduce a norm in the vector space $\,\mathscr V_p\,$
\begin{equation}
\|\,\textnormal E \| = \|\,\textnormal E \|_{\mathscr V_p} \,\bydef \sup_{|\xi|+|\zeta| = 1} |\,\textnormal E(\xi, \zeta) |
\end{equation}
so $\,(\mathscr V_p \;, \|\cdot \| )\,$ becomes a Banach space. The norm of Burkholder function equals
$$
\|\,\mathbf B_p \|\; = \,p^{\ast} \,- 1
$$
 Consider the subsets $\,\mathcal C_p^{\,\smallsmile} \subset \mathscr V_p^{\,\smallsmile}\,$ and $\,\mathcal C_p^{\,\smallfrown} \subset \mathscr V_p^{\,\smallfrown}\,$ of functions whose norm does not exceed $\,p^{\ast} \,- 1 \,$. These are convex sets.
\begin{theorem} \label{ExtremePoint} The Burkholder function $\,\mathbf B_p\,, \,p \neq 2\,,\,$ is an extreme point of $\,\mathcal C_p^{\,\smallfrown}\,$. Similarly, $-\mathbf B_p$ is an extreme point of $\,\mathcal C_p^{\,\smallsmile}\,$.
\end{theorem}
\begin{proof}
Consider a convex combination of $\,\mathbf B_p = \mathbf B_p^{\,\smallfrown}\,$

\begin{equation}\label{Decomposition}
\mathbf B_p  =\sum_{1\leqslant i \leqslant n}   \lambda_{\,i}\,\mathbf E_{\,i} \;,\;\; \lambda_{\,i} > 0\,,\;\lambda _1 + ... + \lambda _n = 1\,, \;\textnormal{where $\,\mathbf E_{\,i} \in \mathcal C_p ^\smallfrown\,$.}
\end{equation}
By Proposition (\ref{ResultOfDecomposition}) there exist positive numbers $\,\theta _i > 0\,$ such that
\begin{equation}\label{DecompositionOutcome}
\mathbf E _i \;\equiv\; \theta _i \,\mathbf B_p\;,\;\;\textnormal{for all} \; i = 1, 2, ... , n\;,\;\;\textnormal{and}\;\; \sum_{1\leqslant i \leqslant n}   \lambda_i\,\theta_i  \;= 1 .
\end{equation}
Computing the norms yields:
$$
\,p^{\ast} \,- 1 \geqslant \|\,\mathbf E _i\,\| =  \,\|\,\mathbf B_p \|\,\theta_i =  (p^{\ast} \,- 1)\, \theta_i
$$
Therefore $\,\theta _i \leqslant 1\,$, for every $\, i = 1, 2, ... , n\,$. On the other hand, in view of $\,\sum_{1\leqslant i \leqslant n}   \lambda_i\,\theta_i  \;= 1\,$ and $\,\lambda _1 + ... + \lambda _n = 1\,$, we have $\,\theta _i = 1\,$, for every $\, i = 1, 2, ... , n\,$. This means that each $ \,\mathbf E _i\,$ equals $\,\mathbf B_p\,$, as desired.
\end{proof}
\section{Burkholder's function is a maximal element}

\begin{theorem} \label{Uniqueness} Among all rank-one concave functions $\,\mathbf E \,:\,\mathbb C \times \mathbb C \rightarrow \mathbb R\,$  that are isotropic and homogeneous of degree $\, p > 1\,$, the function $\, \mathbf B_p(\xi, \zeta)  \,=  \big[|\,\xi| -  ( p^{_\ast}- 1)\, |\,\zeta|  \,\big]\cdot \big[\, |\xi|\, + \,|\zeta|\,\big ] ^ {p-1}\,$ is a maximal one; that is, the inequality
\begin{equation}
 \mathbf B_p(\xi, \zeta)  \,\;\leqslant \;\mathbf E (\xi, \zeta)\,, \;  \;\textnormal{for all}\;\; (\xi,\zeta)  \in \mathbb C \times \mathbb C\,,
\end{equation}
forces $\,\mathbf E\,$ to be equal to $\,\,\mathbf B_p\,$.
\end{theorem}

\begin{proof}
The proof goes through as for Proposition \ref{ResultOfDecomposition}\,, with a slight change. Under the same notation, we
begin with an  energy estimate in the ball $B(0,R)$ (with $R>1$) for the special radial stretchings (7.4)
depending on parameter $\alpha$. Thus

\begin{eqnarray} && \mathscr B_p [f] \leqslant \,\mathscr E [f] \; \leqslant \mathscr E [f^R] \, \;\;\;\;\;\;\;\;\;\;\;\;\;\;\;\;\;\;\;\;\;\;\;\Big (\;\;\textnormal{by Proposition \ref{EnergyRdialLinear} }\;\Big) \, \nonumber \\ &=& \,\pi R^{2-p} \rho^p(R) \,\mathbf E (I_\pm) \;\;\;\;\;\;\;\;\;\;\;\;\;\;\;\;\;\;\;\;\;\;\;\;\;\;\,\Big (\,\textnormal{where}\,\,{\;\,I_+ = \textnormal{id}  \;\;,\, 2 \leqslant p <\infty\,\;\atop \;\,I_- = \overline{\textnormal{id}} \;\; ,\,1<p\,\leqslant \;2\,\,}\,\Big)\, \nonumber \\ & =&      \,\pi R^{2-p} \rho^p(R) \,\mathbf B_p (I_\pm) \; \theta  \;=\;\mathscr B_p [f^R]\,\theta =\,\mathscr B_p [f] \; \theta\;\;\,\;\;\;\;\;\;\;\;\Big(\;\theta \;=\; \frac{\mathbf E (I_\pm)}{\mathbf B_p (I_\pm)} \,\Big)\, \nonumber
\end{eqnarray}

Hence

$$\, \int_{|z|\leqslant R }\mathbf E [Df]\,\textnormal{d} z \;\leqslant\; \theta  \,\int_{|z|\leqslant R }\mathbf B_p [Df]\,\textnormal{d} z .$$

Note that $\,Df (z) \equiv I_\pm \,$ for $ |z| \leqslant 1 \,$ and $\,\mathbf E [I_\pm] = \theta\, \mathbf B_p[I_\pm] \,$. Therefore $\, \int_{|z|\leqslant 1 }\mathbf E[Df]\,\textnormal{d} z \;=\; \theta \,\int_{|z|\leqslant 1}\mathbf B_p [Df]\,\textnormal{d} z .$ The energy inequality reduces to:

\begin{equation}\label{ReducedIdentity2}
\, \int_{1\leqslant|z|\leqslant R }\mathbf E[Df]\,\textnormal{d} z \;\leqslant\; \theta\,\int_{1 \leqslant |z|\leqslant R }\mathbf B_p [Df]\,\textnormal{d} z .
\end{equation}
As before, since we are testing (\ref{ReducedIdentity2})\,with the maps (\ref{Testfunctions})\, using all values of $\; -1 < \alpha < 1\,$,
this results in a point-wise inequality
\begin{equation}\nonumber
\mathbf B_p (\xi, \zeta) \leqslant \mathbf E(\xi, \zeta) \,\;\leqslant\;\theta\;\mathbf B_p (\xi, \zeta)\,,\;\;\;\textnormal{with a constant $\,\theta \in \mathbb R \,$ and all}\;\; \xi\,, \zeta \in \mathbb C\,.
\end{equation}
We must have $\,\theta\,$ equal to 1, because the function $\,\mathbf B_p\,$ attains both strictly positive and strictly negative values. This implies that the first inequality must actually be an equality.
\end{proof}

\section{Local maxima} \label{localmax}
Theorem \ref{AIPS} yields a number of interesting properties for the Burkholder function. 
For instance, under an additional assumption on $\,\rho\,$ the $\,\mathbf B_p\,$-energy, $\,p > 2\,$, assumes its local maximum at the radial stretchings $\,f(z) =  \rho(|z|) \frac{z}{|z|}\,,$ in an explicitly specified neighbourhood of $f$. To see this,  assume that $\,\rho: [0 , 1 ] \rightarrow [0, 1]\,$ is Lipschitz continuous,  $\,\rho(0) = 0\,,\; \rho(1) = 1\,$,  and that for almost every $\, r \in [ 0 , 1 ]\,$ it holds:
\begin{equation}\label{RhoCondition}
 \rho(r) \geqslant r \,\rho\,' (r)\,\geqslant \big(1 - \,2/s\big)\; \rho(r)\,,\;\; \textnormal{for some}\; s > p\,.
\end{equation}
\begin{proposition}\label{LocalMaximum} Consider an  $\,\varepsilon\,$-perturbation of $\,f\,$
\begin{equation}\label{Perturbation}
f^\varepsilon (z) = \rho(|z|) \frac{z}{|z|} \,+\,\varepsilon(z)\;,
\end{equation}
with   $\,\varepsilon\in \mathscr C^1_\circ(\mathbb D) \,$, small enough to satisfy
\begin{equation}\label{Smallepsilon}
 (p-1)\,|\,\varepsilon_{\bar{z}}| \,+\, |\,\varepsilon_z| \leqslant 1 - \frac{p}{s}
 \end{equation}
Then
\begin{equation}\label{LocalInequality}
\mathscr B_p [ f^\varepsilon ] \leqslant \mathscr B_p [ f] \;=\; \mathscr B_p [\textnormal{Id}] = \pi
\end{equation}
\end{proposition}
\begin{proof} The inequality (\ref{LocalInequality})\,would hold if (according to Conjecture  \ref{Quasiconvexity-at-0}) $\,\mathbf B_p\,$ was quasiconcave, by the very definition of quasiconcavity; consequently, condition at (\ref{Smallepsilon}) would be redundant. But we do not know the answer to this conjecture. Fortunately, there is a very satisfactory partial answer; namely, inequality (\ref{LocalInequality}) holds whenever the energy integrand $\,\mathbf B_p (|f^\varepsilon_z| \,,\, | f^\varepsilon_{\bar{z}}\,| ) \,$ is nonnegative  and $\,f^\varepsilon(z) \equiv z\,$ for $\,|z| = 1\,$, see Theorem \ref{AIPS}. Thus, we are reduced to proving the distortion inequality
$$\,\frac{|f^\varepsilon_{\bar{z}} |}{|f^\varepsilon_z| } \leqslant \frac{1}{p-1}$$
The essence of the condition (\ref{RhoCondition}) is the following slightly stronger distortion inequality for the mapping $\,f\,$ :

$$\,\frac{|f_{\bar{z}} |}{|f_z| } = \frac{\rho(r) \,-\, r\rho\,'(r)} {\rho(r) \,+\, r\rho\,'(r)} \leqslant \frac{1}{s-1} < \frac{1}{p-1} \;\;\;\;\;\;\;\;\;\;\textnormal{(\,by\,\,\ref{RhoCondition})}\;.$$
This leaves us a margin for small perturbations of $\,f\,$. Here is how one can exploit this margin.
First the condition $ \, \rho(r) \geqslant r \,\rho\,' (r)\, $ tells us that the function $\,\frac{\rho(r)}{r} \,$ is nonincreasing. Since  $\,\frac{\rho(1)}{1}  = 1\,$, we see that  $\,\frac{\rho(r)}{r} \geqslant 1\,$ and, again by (\ref{RhoCondition}), we obtain
  $$\, \,|f_z| = \frac{1}{2} \left(\frac{\rho(r)}{r}  +  \rho \,'(r) \right) \geqslant \, \frac{1}{2}\left (1 + 1- 2/s\right) \, = \frac{s-1}{s} \;,\; \textnormal{thus}\; 1 \leqslant  \frac{s}{s-1} \,|\,f_z\,| \,.$$
  Next we estimate the derivatives of $\,f^\varepsilon\,$,
  $$
  |\,f^\varepsilon_{\bar{z}}\,| \leqslant |\,f_{\bar{z}}\,|\; +\,|\,\varepsilon_{\bar{z}}\,|\;,\;\;\;\;\;|\,f^\varepsilon_z\,| \geqslant |\,f_z\,|\; - \,|\,\varepsilon_z\,|
  $$
  Hence
  \begin{eqnarray}
(p-1) \,|\,f^\varepsilon_{\bar{z}}\,| \;-\;|\,f^\varepsilon_z\,| &\leqslant & (p-1) \,|\,f_{\bar{z}}\,| \;-\;|\,f_z\,|  \, + \;(p-1)\,|\,\varepsilon_{\bar{z}}| \,+\, |\,\varepsilon_z|\nonumber\\ &\leqslant& \frac{p-1}{s-1} \,|\,f_z| \; - |\,f_z\,| + \left( 1 - \frac{p}{s}\right)  \cdot \frac{s}{s-1} \,|\,f_z\,|\; = 0\,.\nonumber
\end{eqnarray}
as desired.
 \end{proof}

 \section{Radial Mappings as Stationary Solutions}

In order to speak of the Lagrange-Euler equation we have to increase regularity requirements on the integrand and on the mappings in question.
Consider a general isotropic energy functional;
\begin{equation}\label{Isotropic}
\mathscr E[f] = \int_\Omega \textbf{E} \big(|f_z|\,,\, |f_{\bar{z}}|  \big) \,\textrm{d}z\;.
\end{equation}
Here the function $\,\textbf{E} = \textbf{E}( u\,,\, v )\,$ is defined and  continuous on $\,[0 , \infty)\times [0 , \infty)\,$. We assume that $\,\textbf{E} \,$ is  $\,\mathscr C^2\,$-smooth in the open region $\, \mathbb R_+\times\mathbb R_+  = (0 , \infty)\times (0 , \infty)\,$.
A map $\, f \in \mathscr C^1(\Omega)\,$ such that
$$
 \big(\,|f_z(z)| \,,\; |f_{\bar{z}}(z)\,\big) \in \, \mathbb R_+\times\mathbb R_+ \,  \;\quad\; \textrm{for every}\;\; z\in \Omega \subset \mathbb C
$$
is a critical point, or stationary solution, for (\ref{Isotropic}) if for each test function $\,\eta \in \mathscr C_\circ^\infty(\Omega)\,$ it holds
$$
 \frac{\partial}{\partial \,\bar{\tau} }\;\mathscr E[ f + \tau \,\eta ] \,\big|_{\tau = 0}  \;\;\;= \;0\,\;\;\quad\;  \;\;\;\textnormal{ (here $\tau$ is a complex variable)}  $$
 It should be noted that we are using the Cauchy-Riemann derivative $\,\partial/\partial \,\bar{\tau} \,$ in the derivation of the variation of the energy functional. This leads to an integral form of the Euler-Lagrange equation
$$
 \int _\Omega \Big[\,\frac{\partial \textbf{E}}{\partial u} \,\frac{f_z}{|f_z| }\,\,\overline{\eta_z} \;+\; \frac{\partial \textbf{E}}{\partial v} \,\frac{f_{\bar{z}}}{|f_{\bar{z}}| }\,\,\overline{\eta_{\bar{z}}}  \,\;\Big]\,\textrm{d}z \;=\;0
$$
Integration by parts yields a second order divergence type PDE
\begin{equation}\label{Lagrange}
 \Big[\,\frac{\partial \textbf{E}}{\partial u} \,\frac{f_z}{|f_z| } \;\Big]_{\bar{z}} +\; \Big[\,\frac{\partial \textbf{E}}{\partial v} \,\frac{f_{\bar{z}}}{|f_{\bar{z}}| }\;\Big]_z \, \;=\;0
\end{equation}
in the sense of distributions. From now on we assume that $\, f \in \mathscr C^2(\Omega)\,$ and abbreviate the notation for partial derivatives of $\,\textbf{E}\,$ to $\,\textbf{E}_u\,$ and $\,\textbf{E}_v\,$, respectively. Let us also introduce the auxiliary functions:
$$
\alpha = \alpha(z) = \frac{f_z}{|f_z|}\, \in \mathbb S^1 \;\quad\;\textrm{and}\;\;\; \beta = \beta(z) = \frac{f_{\bar{z}}}{|f_{\bar{z}}|}\, \in \mathbb S^1
$$
$$
 u = u(z) = |f_z|\;\quad\;\; v= v(z) = |f_{\bar {z}}|
$$
Upon lengthy though elementary computation the Euler-Lagrange system (\ref{Lagrange}) takes the form
\begin{eqnarray}\label{lagrangeEuler}
\lefteqn{\quad\;\;\; \big( \,\bar{\,\alpha}\,^2\, f_{zz} \;+\; \overline{f_{z\bar{z}}} \,\big)\, \textbf{E}_{uu}  + }\\\nonumber
& & \big(\, 2\, \bar{\alpha} \bar{\beta}\, f_{z\bar{z}}\;+ \;\bar{\alpha} \beta \; \overline{f{_{\bar{z}\bar{z}}}}\;+\; \alpha \bar{\beta} \;\overline{f_{zz}}\,\big) \,\textbf{E}_{uv}+\\\nonumber
& & \big(\,\bar{\beta}\,^2\; f_{\bar{z}\bar{z}} \;+\; \overline{f_{z\bar{z}}}  \,\big) \,\textbf{E}_{vv}+\\\nonumber
& & \big(\,\overline{f_{z\bar{z}}}\;-\;\bar{\;\alpha}\,^2\, f_{zz}    \,\big) \,u^{-1}\,\textbf{E}_u+\\\nonumber
& & \big(\,\overline{f_{z\bar{z}}}\;-\;\bar{\beta}\,^2\; f_{\bar{z}\bar{z}} \,\big) \,v^{-1}\,\textbf{E}_v \;\quad\quad\quad =\;0\nonumber
\end{eqnarray}
The question arises when a radial stretching
$$\,f(z) = \rho(\,|z - a|\,) \;\frac{z -a }{|z -a |}\;\;\;+\;\; b\,$$
 satisfies this system (\ref{lagrangeEuler}).  We need only examine the case $\,a=0\,$ and $\,b = 0\,$. Recall formulas for the derivatives:
\begin{equation}\label{ComplexDerivatives}
  f_z(z)   = \frac{1}{2} \Big(\, \dot{\rho}(|z|)\;+\frac{\rho(|z|)}{|z|}\;\Big )\;\quad\;f_{\bar{z}}(z)   = \frac{1}{2} \Big(\, \dot{\rho}(|z|)\;-\;\frac{\rho(|z|)}{|z|}\;\Big )\frac{z}{\bar{z}}
 \end{equation}
As mentioned before, the Euler-Lagrange equation requires $\,\mathscr C^2$-regularity of $\,f\,$. Because of this, we assume that $\,\ddot{\rho}\,$ is continuous. Now, further differentiation of (\ref{ComplexDerivatives}) gives second order derivatives
\begin{eqnarray}\label{E-L}
\lefteqn{\quad\;\;\; 4 f_{zz} \,= \, \Big(\,\ddot{\rho}\;+\;\frac{\dot{\rho}}{|z|}\;-\; \frac{\rho}{|z|^2}\;\Big)\;\frac{\bar{z}}{|z|}}\\\nonumber
& & 4 f_{z\bar{z}} \,= \, \Big(\,\ddot{\rho}\;+\;\frac{\dot{\rho}}{|z|}\;-\; \frac{\rho}{|z|^2}\;\Big)\;\frac{z}{|z|}\\\nonumber
& & 4 f_{\bar{z}\bar{z}} \,= \, \Big(\,\ddot{\rho}\;-\;\frac{3\,\dot{\rho}}{|z|}\;+\; \frac{3\,\rho}{|z|^2}\;\Big)\;\frac{z^3}{|z|^3}\\\nonumber
\end{eqnarray}
For the results in this section we further assume that
 $$\,\rho(|z|) \; > \;|z|\, \dot{\rho}(|z|)\;$$ and hence $\,\alpha \equiv 1\,$ and $\,\beta = \,-\,z/\bar{z}\,$. The Euler-Lagrange equation takes the form
\begin{equation}\label{leEquation}
\Big(\ddot{\rho}\;+\;\frac{\dot{\rho}}{|z|}\;-\; \frac{\rho}{|z|^2}\Big)\textbf{E}_{uu} \;-\,\; 2\,{\ddot{\rho}}\, \textbf{E}_{uv}\;+\; \Big(\ddot{\rho}\;-\;\frac{\dot{\rho}}{|z|}\;+\; \frac{\rho}{|z|^2}\Big)\textbf{E}_{vv} \;=\, \frac{4}{|z|}\, \textbf{E}_v
\end{equation}
Note the absence of the term $\,\textbf{E}_u\,$. Indeed,  the variables $z$ and $\bar{z}$ play uneven role in our considerations.
For a radial mapping we have $\, 2\,v = 2\,|f_{\bar{z}}(z)|\; =  \big(\frac{\rho}{|z|} \,-\,\dot{\rho}\big)\,$, so the equation (\ref{leEquation}) takes the form:
\begin{equation}\label{EulLagr}
\big(\textbf{E}_{uu} \;- \,2\;\textbf{E}_{uv} \;+ \textbf{E}_{vv}\big)\;|z|\,\ddot{\rho}\;\;=\;\; \,2\,\big(\textbf{E}_{uu} \;-\; \textbf{E}_{vv} \;+\;2\,v^{-1}\textbf {E}_v \big)\,v
\end{equation}\label{ELsystem}
We shall now take a quick look at the Euler-Lagrange equation for the Burkholder energy $ \,\mathscr B_p [f] \;=\;\int_{\Omega} \big[\,|f_{z}|\;-\;(p-1) |f_{\bar{z}}| \;\big]\,\cdot\, \big[\;|f_z|\;+\;|f_{\bar{z}}|\;\big]^{p-1}\, \textrm d z \,$. Direct computation shows that the integrand
$$
 \textbf{E} = \textbf{E}(u,\,v) = [u \,-(p-1)\,v ]\,\cdot\, [u \,+\,v ] ^{p-1}\;
$$
satisfies the following system of partial differential equations
\begin{eqnarray}\label{PDEs}
  \left\{\begin{array}{l}
 \textbf{E}_{uu} \;-\,2\;\textbf{E}_{uv} \;+ \textbf{E}_{vv}\;=\; 0 \\  \\ \textbf{E}_{uu} \;-\; \textbf{E}_{vv} \; =\;- \;2\,v^{-1}\,\textbf {E}_v
\end{array}\right.
\end{eqnarray}
\begin{corollary}
The radial stretching $\,f\,$ (as specified above) is a critical point of the Burkholder energy functional
$ \;\mathscr B_p [f]\;$.
\end{corollary}

It is purely theoretical but still interesting to know which  variational integrals admit such radial mappings among their stationary solutions.
We shall see that only Burkholder integrals fulfil this requirement. To this effect we observe that the equation (\ref{EulLagr}), being satisfied for varied radial mappings, yields the system of PDEs in (\ref{PDEs}).
Indeed, let us view the terms in (\ref{EulLagr}) as  functions in three variables $\,|z|,\, \rho\,$ and $\, \dot{\rho}\,$, plus  linear dependence on   $\,\ddot{\rho}\,$. When the radial maps run over the admissible class, the term $\,\ddot{\rho}\,$ varies point-wise independently of the remaining three variables $\,|z|,\, \rho\,$ and $\, \dot{\rho}\,$ . This is possible only when both equations in (\ref{PDEs}) are satisfied.\\
 Now we are left with the task of solving the system (\ref{PDEs}). Here the second equation is reminiscent of the planar wave equation, suggesting to change variables in the following fashion:
$$
\xi = u \,+\,v\,, \;\;\;\; \zeta  = u\,-\,v  \;, \;\;\textrm{so}\;;\;\;2\,u = \xi +\zeta \;\;\textrm{and} \;\;\;2\,v = \xi - \,\zeta
$$
Now we express the integrand $\,\textbf{E}\,$ in the form $\,\textbf{E} (u,v) =  \Phi(\xi, \zeta)\,$.
The system (\ref{PDEs}) translates into the following equations for $\,\Phi\,$

\begin{eqnarray}\label{PDE}
  \left\{\begin{array}{l}
 \Phi_{\zeta\,\zeta}=\; 0 \\  \\ (\xi - \zeta) \Phi_{\xi\,\zeta} \;=\; \Phi_\zeta \;-\;\Phi_\xi
\end{array}\right.
\end{eqnarray}
Thus $\,\Phi\,$ is affine in the $\,\zeta$-variable; precisely, $\, \Phi(\xi,\,\zeta) \,=\, A(\xi)\,\zeta \;+\; B(\xi)\,$. Then the second equation yields the following ODE for the coefficients $\,A(\xi)\,$ and $\, B(\xi)\,$:
\begin{equation}\label{ODE}
 \dot{B}(\xi)\;=\;\; A(\xi)\;-\;\xi\, \dot{A}(\xi)
\end{equation}

Finally, suppose (like in the Burkholder's functional) that $\,\Phi\,$ is homogeneous of degree $\,p \, $. Thus, up to a constant factor, $\, A(\xi)\,=\, p \, \xi^{p-1}\,$. Then Equation (\ref{ODE}) yields $\,B(\xi)\,= \,(2-p)\, \xi^{p}\,$. Hence $\, \Phi(\xi,\,\zeta)\,=\, [\,p\; \zeta\; +\; (2-p)\,\xi\,] \;\xi^{p-1}\,$. Having in mind that $\,\xi = u + v\,$ and $\, \zeta  = u - v\,$, we return to $\,u , v\,$-variables. It results in the Burkholder function $\,\textbf{E}(u,\,v)\,= [u - (p-1)\, v]\cdot [u +\, v]^{p-1} \,$.

\begin{corollary}
The only isotropic $\,p\,$-homogeneous variational integrals which hold all radial mappings (of type specified above)  among their stationary solutions are the scalar multiples of $\,\mathscr B_p\,[f] \,$.
\end{corollary}

\section{Quasiconcavity at zero  versus quasiconcavity at $\,A \in\mathbb R^{2\times 2}\,$}
Let us begin with an example:
 \begin{example} \label{MaximumEnergyExample}
The following function belongs to the Sobolev space $\,\mathscr W^{1,p}(\mathbb C)\,$ for every $\,1 < p < \infty\,$ and its $\,\mathbf B_p\,$-energy equals zero.
\begin{displaymath}
 f(z)\; = \; \left\{\begin{array}{ll}
 z\;\; & \textrm{if $\,|z| \leqslant R\,$}  \\
    \frac{R^2}{\bar{z}} & \textrm{if $\,|z| \geqslant R\,$}
\end{array} \right.
\end{displaymath}
\end{example}Indeed, we have
\begin{eqnarray}
\mathscr B_p\,[f] &= & \int_\mathbb C  \mathbf B_p ( f_z , f_{\bar{z}} ) \, \textrm d z\;=\; \int_{|z| \leqslant R} \textrm d z \; - \; (p-1) R^{2p} \int_{|z| \geqslant R } \frac{\textrm d z }{|z|^{2p}}\nonumber\\ &=& \pi R^2 \;-\; (p-1) R^{2p} \,\frac{\pi R^{2-2p}}{p-1}\; = 0\,.\nonumber
\end{eqnarray}
In view of Conjecture  \ref{Quasiconvexity-at-0} one may expect $\,f\,$ to have maximum energy (equal to zero) within the class $\,\mathscr W^{1,p}(\mathbb C)\,$. This example gains additional interest if we can answer in the affirmative the following
\begin{question} \label{BigQuestion} Given a linear map $\, z \mapsto a z + b {\bar{z}}\,$, does there exist a function $\, f \in \mathscr W^{1,p} (\mathbb C)\,$  such that
\begin{displaymath}
  \left\{\begin{array}{ll}
 f(z) = a z + b {\bar{z}} \;,\;  \textrm{in some nonempty domain  $\,\Omega \subset \mathbb C\,$}  \nonumber\\
 $\,$ \\
 \mathscr B_p\,[f] \;=\; \int_\mathbb C  \mathbf B_p ( f_z , f_{\bar{z}} ) \, \textnormal{d}z\;= 0
\end{array}\nonumber \right.
\end{displaymath}
\end{question}
In other words:
\begin{equation}\label{extremalsAtZero}
\int_{\mathbb C \setminus \Omega} \mathbf B_p ( f_z , f_{\bar{z}} ) \, \textnormal{d}z\;= \; - \mathbf B_p(a, b) \, |\Omega|
\end{equation}
Now quasiconcavity of $\,\mathbf B_p\,$ at zero would tell us that $\,- \mathbf B_p(a, b) \, |\Omega|\,$ is the maximum energy among Sobolev mappings in $\,\mathscr W^{1,p}(\mathbb C \setminus \Omega)\,$ which agree with $\, a z + \, b \bar{z}\,$ on $\,\partial \Omega\,$.
Question \ref{BigQuestion}\, has yet another interesting effect.
\begin{proposition}
Accept that Burkholder function is quasiconcave at zero and that a linear map $\,A z =a z + \, b \bar{z}\,$ in a domain $\,\Omega\,$ has been found to admit an extension satisfying (\ref{extremalsAtZero}).
Then $\,\mathbf B_p\,$ is quasiconcave at $\,A\,$.
\end{proposition}
\begin{proof} Let $\,\varphi \in \mathscr C_\circ^\infty(\Omega)\,$ be any test mapping. We need to show that the following integral is nonpositive,
\begin{eqnarray}
&&\int_\mathbb C \big[\mathbf B_p(A + D\varphi) - \mathbf B_p(A) \big] \; = \int_\Omega \big[\mathbf B_p(A + D\varphi) - \mathbf B_p(A) \big] \nonumber\\\;& = & \int_\Omega \mathbf B_p(A + D\varphi) + \int_{\mathbb C \setminus \Omega}  \mathbf B_p(Df)   =  \int_{\mathbb C}  \mathbf B_p(DF) \;\leqslant 0\,.\nonumber
\end{eqnarray}
where
\begin{displaymath}
  F(z) = \left\{\begin{array}{ll}
 a z + b {\bar{z}} \, + \varphi(z) \;\;  \textrm{in the domain  $\,\Omega \subset \mathbb C\,$}  \nonumber\\
 $\,$ \\
 f(z)  \,\; \textnormal{in}\;\mathbb C \setminus \Omega
\end{array}\nonumber \right.
\end{displaymath}
The latter inequality follows since $\, F \in \mathscr W^{1,p}(\mathbb C)\,$ and $\,\mathbf B_p\,$ was assumed to be quasiconcave at zero.
\end{proof}
Now it follows from Example \ref{MaximumEnergyExample} that
\begin{corollary} Quasiconcavity of $\,\mathbf B_p\,$ at zero would imply quasiconcavity at the identity matrix.
\end{corollary}

\bigskip

We believe that the presented advances (including some of the conditional statements for the Burkholder functions) will convince the interested readers  of the intricate nature of computing the $\,p\,$-norms of the Beurling Transform.

\bibliographystyle{amsplain}

\end{document}